\documentclass{IEEEtran}
\usepackage{cite}
\usepackage{amsmath,amssymb,amsfonts}
\usepackage{algorithmic}
\usepackage{graphicx}
\usepackage{textcomp}
\usepackage{mathtools}
\newcommand{\scal}[2]{\left\langle {#1} \middle| {#2} \right\rangle}
\usepackage{subcaption}
\usepackage{empheq}

\newtheorem{thm}{Theorem}
\newtheorem{defn}{Definition}
\newtheorem{lem}{Lemma}
\newtheorem{cor}{Corollary}
\newtheorem{remm}{Remark}
\newtheorem{exmp}{Example}
\newtheorem{pb}{Problem}
\newenvironment{proof}{{\it Proof :~}}{\hfill$\square$\\}
\newenvironment{rem}{\begin{remm}\rm }{\hfill \hspace*{1pt} \hfill $\lrcorner$\end{remm}}
\usepackage{pgf,tikz,pgfplots}
\pgfplotsset{compat=1.15}
\usepackage{mathrsfs}
\usetikzlibrary{arrows}
\usepackage{verbatim}

\begin{document}
\title{\LARGE Instability conditions for reaction-diffusion-ODE systems}
\author{Mathieu Bajodek, Hugo Lhachemi, Giorgio Valmorbida
%\thanks{}
\thanks{The authors are with Universit\'e Paris Saclay, CNRS, CentraleSupelec, INRIA, Laboratoire des signaux et syst\`emes, 91190 Gif Sur Yvette, France (e-mail: mathieu.bajodek@centralesupelec.fr).}}

\maketitle

\begin{abstract}
This paper analyzes the stability of a reaction-diffusion equation coupled with a finite-dimensional controller through Dirichlet boundary input and Neumann boundary output. Going against the flow, we intend to propose numerical certificates of instability for such interconnections. From one side, using spectral methods, an analytical condition based on root locus analysis can determine the instability regions in the parameters space and can sometimes be tested. On the other side, using Lyapunov direct and converse approaches, two sufficient conditions of instability are established in terms of linear matrix inequalities.
The novelties lie both in the type of system studied and in the methods used. The numerical results demonstrate the performance of the different criteria set up in this paper and allow us to conjecture that these conditions seem to be necessary and sufficient.
\end{abstract}


\section{Introduction}
\label{sec:introduction}
Physical phenomena of heat, particles, or electric charge flows are modeled by reaction-diffusion equations~\cite{Plawsky2009}. In these systems, the inputs and outputs often appear at the boundaries. Therefore, to stabilize or regulate the state of a reaction-diffusion system, boundary measurements and actuation are imposed. Moreover, for practical reasons, stabilizing control laws are often required to be finite-dimensional and linear. The interconnection between the system and the controller thus leads to a non-standard linear infinite-dimensional system, which makes its stability analysis a tough task~\cite{Suzuki2017,Mironchenko2019}. This paper studies this class of interconnections aiming at developing numerical methods for the stability analysis and proposes conditions allowing us to conclude on the stability in the space of the system's parameters, namely the reaction and diffusion coefficients as well as the finite-dimensional terms.
%They are subject to boundary inputs and outputs. In control theory, we are interested in the design of finite-dimensional control laws, which stabilize the closed-loop system. 
%This coupled system is a non-standard linear infinite-dimensional system, which makes its stability analysis difficult~\cite{Mironchenko2019}. 
%Our paper purpose is to develop numerical tools to analyze its stability with respect to the actuators or measuring device parameters but also the reaction or diffusion coefficients.

Sufficient conditions for the stability of interconnections between a partial differential equation (PDE) and an ordinary differential equation (ODE) can be obtained by structuring Lyapunov functional as the sum of two quadratic terms, one associated to the infinite-dimensional part~\cite{Peet2006} and another one related to the ODE dynamics~\cite{Prieur2012}. Such an approach, imposing negative derivatives along the trajectories of the coupled system, leads to pessimistic stability estimates. To these limited Lyapunov functional structures, crossed terms between PDE and ODE variables can be introduced with particular parametrizations of the Lyapunov functional~\cite{Peet2019} using Legendre~\cite{Baudouin2019}, spectral~\cite{Prieur2019} or Fourier~\cite{Katz2020} terms. In~\cite{Karafyllis2019}, an input-to-state analysis based on the small-gain theorem is also presented.
These techniques reduce conservatism and simplify the design approach, although they remain only sufficient conditions for stability.

On a different vein, backstepping methods aim to design state feedback controllers with boundary actuation by fixing an inherently stable target system~\cite{Krstic2009,Tang2011}. Nevertheless, the panel of target systems is limited and the technique, in general, requires an additional step of discretization of the control~\cite{Auriol2019} and may lack the robustness addressed by dynamical finite-dimensional controllers~\cite{Karafyllis2019}. To get free of instability phenomena or loss of robustness, an anticipation stage to the design process is needed and studied in this paper.
%as the ones studied in this paper.
%Most of the time, researchers objectives are dedicated to sufficient stability conditions, ideally not too pessimistic. Typically, using the sum of the Lyapunov functionals associated to the ordinary (ODE) and partial (PDE) differential equations and imposing negative derivatives along the trajectories of the coupled system, already provide a raw stability criterion. It is moreover with the help of this type of inherently stable target systems that backstepping controllers are designed~\cite{Krstic2009,Tang2011}.
%Nevertheless, this target systems panel is conservative. It limits the selection and prevent the realization of simpler and more effective controllers. %choice of what ois possible to do
%Approaches based on approximated Lyapunov functionals have also been introduced, by enriching the finite-dimensional state with Legendre~\cite{Baudouin2019}, spectral~\cite{Prieur2019} or Fourier~\cite{Katz2020} coefficients. In~\cite{Meurer2022,Karafyllis2019}, an input-to-state analysis based on the small-gain theorem is also presented. These techniques reduce conservatism, simplify the design approach, although they remain only sufficient conditions for stability.

In this paper, we approach the stability analysis of reaction-diffusion PDE and ODE interconnections from a different angle by proposing sufficient conditions of instability for the system. Therefore, we establish conditions to identify unstable systems thus allowing us to determine regions of parameters yielding unstable trajectories~\cite{Casten1978,Suzuki2017}. Combined with sufficient stability conditions, our result enables us to obtain inner and outer approximations of the stability regions in the state of parameters. The study of instability conditions has been considered for transport PDE coupled with an ODE (time-delay systems)~\cite{Mondie2017,Mondie2022,Sipahi2011,Oliveira1994}, which inspired us to include the case of reaction-diffusion PDE and ODE interconnected systems. Our study will look in particular at spectral or quasi-spectral projection methods.
%Our angle of attack is fundamentally opposed in the sense that we provide sufficient condition of instability. Inspired by promising results for the case of time-delay systems~\cite{Mondie2017,Mondie2022,Sipahi2011}, we wonder if it is possible to tackle the case of ODE-reaction-diffusion systems.
%In this regard, our objective is to identify unstable systems. These considerations allow to guarantee that certain systems cannot be considered as targets, to prevent destabilizing phenomena of a closed-loop system and to certify that stability is not independent of some parameters. The last interest lies in the duality of the stability problem.
%Combined with sufficient stability approaches, we can obtain inner and outer estimates of the stability regions. This could even lead to necessary and sufficient conditions.

Section~II presents our linear reaction-diffusion-ODE system and its characteristics. %ODE-reaction-diffusion-ODE system.
From one side, Section~\ref{sec:spectral} deals with spectral analysis. 
In the Laplace domain, using a Riesz decomposition of our operator, the stability relies on the location of the roots of the characteristic equation~\cite{Curtain2020} (as for time-delay systems~\cite{Sipahi2011}). On the other side, Section~\ref{sec:lyapunov} uses Lyapunov analysis. In the time domain, the existence of a positive Lyapunov operator is crucial~\cite{Datko1970} (as for time-delay systems~\cite{Kharitonov2006,Mondie2022}). Direct and converse Lyapunov instability conditions are then proposed with the help of projections on sub-spaces of the infinite-dimensional state space. Section~V is finally devoted to two examples. Particular attention is paid to the simple case of scalar systems for illustrative purposes.

\textit{Notation:}  In this paper, the set of natural, real, complex numbers, real matrices of size $n\times m$ and of symmetric positive definite matrices of size $n$ are denoted by $\mathbb{N}$, $\mathbb{R}$, $\mathbb{C}$, $\mathbb{R}^{n \times m}$ and $\mathbb{S}^{n}_+$, respectively. For any $s\in\mathbb{C}$, $\mathrm{Re}(s)$ and $\mathrm{Im}(s)$ represent its real and imaginary parts and $\mathbf{i}=\sqrt{-1}$. The notation $\mathbf{e}_j^n$ stands for the $j$-th vector of the canonical basis of $\mathbb{R}^n$. For any square matrix $M$, $M\succ 0$ means that $M$ belongs to $\mathbb{S}_n^+$ and $\mathrm{He}(M)=M+M^\top$, where $M^\top$ is the transpose of matrix $M$. Denote also its determinant $\mathrm{det}(M)$, adjugate $\mathrm{adj}(M)$ and kernel $\mathrm{ker}(M)$. Moreover, $\mathcal{A}^\ast$ will be used for the adjoint of operator $\mathcal{A}$. We also consider functions $\mathrm{cosh}(\sigma)=\frac{e^{\sigma}+e^{-\sigma}}{2}$, $\mathrm{sinh}(\sigma)=\frac{e^{\sigma}-e^{-\sigma}}{2}$ and $\mathrm{sinhc}(\sigma)=\frac{e^{\sigma}-e^{-\sigma}}{2\sigma}$. We finally set $\mathcal{H}:=\mathbb{R}^n\times L^2(a,b)$ and $\mathcal{H}^2:=\mathbb{R}^n\times H^2(a,b)$, where $L^2$ is the space of square integrable functions and $H^2$ the second order Sobolev space. In $\mathcal{H}$, define the scalar product by $\scal{\begin{bsmallmatrix}x_1\\z_1\end{bsmallmatrix}}{\begin{bsmallmatrix}x_2\\z_2\end{bsmallmatrix}} \! =\!x_1^\top x_2+\int_{a}^b\! z_1^\top(\theta)z_2(\theta)\mathrm{d}\theta$ and the associated norm $\left\lVert\begin{bsmallmatrix}x\\z\end{bsmallmatrix}\right\rVert^2\! =\!|x|^2+\int_{a}^b |z(\theta)|^2\mathrm{d}\theta$, where $|\cdot|$ is the Euclidean norm. Lastly, introduce the Kronecker function $\delta_{\theta_p}(z):=\left\{\begin{array}{ccc}0&\text{if}&\theta\neq \theta_p\\z(\theta_p)&\text{if}&\theta=\theta_p\end{array}\right.$ and the notation $\mathrm{span}(\mathcal{S})$ for linear combinations of the vectors in the set $\mathcal{S}$.

\section{Problem statement}

\subsection{Reaction-diffusion and ODE interconnected system}

Consider one reaction-diffusion equation interconnected through the boundaries to a set of $n_x$ ODEs
\begin{equation}\label{eq:system}
\left\{
\begin{aligned}
    \dot{x}(t) &= A x(t) + B \partial_{\theta}z(t,\theta_o),\\
    \partial_{t}z(t,\theta) &= (\nu\partial_{\theta\theta}+ \lambda) z(t,\theta),\quad \forall \theta\in(0,\theta_i),\\
    \begin{bsmallmatrix} z(t,0)\\z(t,\theta_i)\end{bsmallmatrix} &=  \begin{bsmallmatrix}Cx(t)\\0\end{bsmallmatrix},
\end{aligned}
\right.
\end{equation}
for all $t\geq 0$, where matrices $A\in\mathbb{R}^{n_x\times n_x}$, $B\in\mathbb{R}^{n_x\times 1}$, $C\in\mathbb{R}^{1\times n_x}$ and where scalars $\lambda\in\mathbb{R}$, $\nu>0$, $\theta_i>0$ and $\theta_o\in[0,\theta_i]$.

This interconnection is representative of a control loop of a reaction-diffusion system, where the PDE part is the plant and the ODE part corresponds to the dynamics of the controller. We associate the following linear operator to system~\eqref{eq:system}
\begin{equation}\label{eq:A}
    \mathcal{A} = \begin{bmatrix} A & B\delta_{\theta_o}\partial_\theta\\ 0 & \nu\partial_{\theta\theta}+\lambda\end{bmatrix},
\end{equation}
on the domain $\mathcal{D}$ given by
\begin{equation}\label{eq:domains}
%\begin{aligned}
    \mathcal{D} := \left\{ \begin{bsmallmatrix}x\\z\end{bsmallmatrix}\in\mathbb{R}^{n_x}\!\times\! H^2(0,\theta_i)\,|\, \begin{bsmallmatrix}C&-\delta_0\\0&-\delta_{\theta_i}\end{bsmallmatrix}\begin{bsmallmatrix}x\\z\end{bsmallmatrix}=0\right\}.
\end{equation}

In the Laplace domain, an irrational transfer function can describe the reaction-diffusion part. Indeed, considering zero initial conditions, we have
\begin{equation}
\left\{
    \begin{aligned}
    sZ(s,\theta) &= (\nu\partial_{\theta\theta}+\lambda)Z(s,\theta),\quad\forall\theta\in(0,\theta_i),\\
    \begin{bsmallmatrix}Z(s,0)\\Z(s,\theta_i)\end{bsmallmatrix} &= \begin{bsmallmatrix}CX(s)\\0\end{bsmallmatrix}.
    \end{aligned}
\right.
\end{equation}
Solving this reaction-diffusion equation with respect to the Laplace variable $s$, the distributed transfer function from $CX(s)$ to $Z(s,\theta)$ is given by
\begin{equation}\label{eq:G}
\begin{aligned}
    G(s,\theta) \!&=\! \begin{bmatrix}e^{\sqrt{\frac{s-\lambda}{\nu}}\theta}\\e^{-\sqrt{\frac{s-\lambda}{\nu}}\theta} \end{bmatrix}^{\!\top} \! \begin{bmatrix} 1 & 1 \\ e^{\sqrt{\frac{s-\lambda}{\nu}}\theta_i} & e^{-\sqrt{\frac{s-\lambda}{\nu}}\theta_i} \end{bmatrix}^{-1}\begin{bmatrix}1\\0\end{bmatrix},\\
    \!&=\! \frac{\sinh\!\left(\!\sqrt{\frac{s-\lambda}{\nu}}(\theta_i-\theta)\!\right)}{\mathrm{sinh}\!\left(\!\sqrt{\frac{s-\lambda}{\nu}}\theta_i\!\right)\!}, \quad \forall \theta\in[0,\theta_i],
\end{aligned}
\end{equation}
for all $s\in\mathbb{C}\backslash \left\{-\nu(\frac{k\pi}{\theta_i})^2+\lambda\right\}_{k\in\mathbb{N}}$ leading to the following transfer function from the input $CX(s)$ to the output $\partial_\theta Z(s,\theta_o)$ \begin{equation}\label{eq:H}
%\begin{aligned}
    H(s) = \partial_\theta G(s,\theta_o)
    %&=\scriptscriptstyle{\sqrt{\frac{s-\lambda}{\nu}}} \displaystyle{\begin{bsmallmatrix}e^{\sqrt{\frac{s-\lambda}{\nu}}\theta_o}\\-e^{-\sqrt{\frac{s-\lambda}{\nu}}\theta_o}\end{bsmallmatrix}^{\!\top}\!\! \begin{bsmallmatrix} 1 & 1 \\\!e^{\sqrt{\frac{s-\lambda}{\nu}}\theta_i}\! & \!e^{-\sqrt{\frac{s-\lambda}{\nu}}\theta_i}\! \end{bsmallmatrix}^{-1}\! \begin{bsmallmatrix}1\\0\end{bsmallmatrix} },\\
    = -\frac{\cosh\!\left(\!\sqrt{\frac{s-\lambda}{\nu}}(\theta_i-\theta_o)\!\right)}{\theta_i\mathrm{sinhc}\!\left(\!\sqrt{\frac{s-\lambda}{\nu}}\theta_i\!\right)\!}.
%\end{aligned}
\end{equation}

Let $\Delta(s) := \mathrm{det}(sI_{n_x} - A - B H(s) C)$. The point spectrum of operator $\mathcal{A}$ are solutions to
\begin{equation}\label{eq:charac}
    \Delta(s) = 0.
\end{equation}

\begin{rem}
Compared to time-delay systems~\cite{Sipahi2011}, the transfer function of the transport equation $e^{-hs}$ (or delay $h>0$) is the irrational transfer function $H(s)$ in~\eqref{eq:H}, which is holomorphic on the set $\mathbb{C}\backslash \left\{-\nu(\frac{k\pi}{\theta_i})^2+\lambda\right\}_{k\in\mathbb{N}}$, namely a meromorphic function. 
\end{rem}

\subsection{Riesz decomposition}
%Well-posedness and 

%Let us now focus on the system's operator properties in $\mathcal{H}$
We focus here on the modal decomposition of the operator $\mathcal{A}$ to deduce the existence and analytic properties of the semigroup generated by~$\mathcal{A}$ on the infinite-dimensional state space $\mathcal{H}:=\mathbb{R}^{n_x}\times L^2(0,\theta_i)$.\\

\begin{lem}
The point spectrum of operator $\mathcal{A}$ in~\eqref{eq:A}, namely the roots of $\Delta(s)$ as in~\eqref{eq:charac}, are isolated and of finite algebraic multiplicity.
\end{lem}
\begin{proof}
The proof is given in~\cite[Lemma~2]{Zhao2019} or~\cite[Lemma~1]{Dlala2022}. It relies on the existence of a sufficiently large scalar $\mu$ such that $(\mu - \mathcal{A})^{-1}$ exists and is compact in $\mathcal{H}$.
\end{proof}

The point spectrum of $\mathcal{A}$ will be denoted by $\{s_k\}_{k\in\mathbb{N}}$ in the sequel. 

\begin{rem}
As a consequence of Lemma~1, there is a finite number of roots $\{s_k\}_{k\in\mathbb{N}}$ contained in any compact subset of $\mathbb{C}$.
\end{rem}
%\begin{rem}
%According to~\cite[Section 6, Theorem~3.1]{Pazy2012}, and after a change of variable bringing the boundary conditions in the domain, system~\eqref{eq:system} is well-posed. The solution~$\begin{bsmallmatrix}x\\z\end{bsmallmatrix}$ is unique and continuous from $[0,\infty)$ to $\mathcal{H}$.
%\end{rem}

\begin{lem}\label{lem:riesz}
There is a set of generalized characteristic functions\footnote{A generalized characteristic function $\mathcal{F}_k$ associated to the characteristic root $s_k$ is non null and satisfy $(s_k-\mathcal{A})^{\delta}\mathcal{F}_k=0$, for a positive integer $\delta$.} $\{\mathcal{F}_k\}_{k\in\mathbb{N}}$ of $\mathcal{A}$ in~\eqref{eq:A}, which forms a Riesz basis for $\mathcal{H}$.%$:=\mathbb{R}^{n_x}\!\times\! L^2(0,\theta_i)$.
\end{lem}
\begin{proof}%, which modify the first modes and only slightly modify asymptotically the reaction-diffusion equation
    The proof is similar to~\cite[Theorem~1]{Zhao2019}. It consists in considering the ODE as an external perturbation of the PDE. The characteristic roots of~\eqref{eq:system} are solution of~\eqref{eq:charac} and verify
    \begin{equation}\label{eq:charac2}
        \mathrm{sinhc}\!\left(\!\sqrt{\frac{s-\lambda}{\nu}}\theta_i\!\right)\!=\! R(s)\cosh\!\left(\!\sqrt{\frac{s-\lambda}{\nu}}(\theta_i-\theta_o)\!\right),
    \end{equation}
    %The calculations to obtain~\eqref{eq:charac2} are similar to~\cite[Appendix~1]{Zhao2019}.
    for some rational fraction $R(s)$, whose numerator and denominator are of degrees $n_x-1$ and $n_x$ respectively. The detailed expression of $R(s)$ can be obtained similarly to~\cite[Appendix~1]{Zhao2019}, but is not needed here. Only the property $R(s)=\underset{s\to\infty}{\mathcal{O}}(\frac{1}{s})$ will be used. Denoting $\sigma =\sqrt{\frac{s-\lambda}{\nu}}\theta_i$, we obtain
    \begin{equation}\label{eq:charac3}
        (e^{\sigma}-e^{-\sigma}) \!=\! \sigma(e^{\sigma}+e^{-\sigma})R\!\left(\!\nu(\sigma/\theta_i)^2+\lambda\!\right)\!\frac{\cosh\!\left(\!\sigma\frac{\theta_i-\theta_o}{\theta_i}\!\right)}{\cosh(\sigma)}.
    \end{equation}
    Hence, arranging the terms in~\eqref{eq:charac3} gives
    \begin{equation}\label{eq:charac4}
        e^{2\sigma} = \frac{1+\sigma R\!\left(\!\nu(\sigma/\theta_i)^2+\lambda\!\right)\!\frac{\cosh\!\left(\!\sigma\frac{\theta_i-\theta_o}{\theta_i}\!\right)}{\cosh(\sigma)}}{\displaystyle 1-\sigma R\!\left(\!\nu(\sigma/\theta_i)^2+\lambda\!\right)\!\frac{\cosh\!\left(\!\sigma\frac{\theta_i-\theta_o}{\theta_i}\!\right)}{\cosh(\sigma)}},
    \end{equation}
    and letting $\sigma\to\infty$ yields
    \begin{equation}\label{eq:charac5}
     e^{2\sigma} = 1+\underset{\sigma\to \infty}{\mathcal{O}}\left(\frac{1}{\sigma}\right).
    \end{equation}
    From Rouché's theorem~\cite[Theorem 5.3.8]{Conway1978}, we show that there exists an integer $n$ from which the roots $\{\sigma_k\}_{k\geq n}$ of~\eqref{eq:charac} are algebraically and geometrically simple and expressed as
    \begin{equation}\label{eq:assymp1}
        \sigma_k = \sqrt{\frac{s_k-\lambda}{\nu}}\theta_i = \mathbf{i}k\pi + \underset{k\to\infty}{\mathcal{O}}\left(\frac{1}{k}\right).
    \end{equation}
    Following Appendix~\ref{app1}, the corresponding normalized characteristic functions $\mathcal{S}_F:=\{F_k\}_{k\geq n}$ of $\mathcal{A}$ in $\mathcal{D}$ are expressed as
    \begin{equation}\label{eq:assymp2}
        F_k := \begin{bsmallmatrix} \!\mathrm{adj}\left((\nu(\sigma_k/\theta_i)^2+\lambda)I_{n_x}-A\right)B \frac{\mathbf{i}\sinh(\sigma_k)}{C\mathrm{adj}\left((\nu(\sigma_k/\theta_i)^2+\lambda)I_{n_x}-A\right)B} \! \\ \mathbf{i}\sinh\left(\sigma_k\frac{\theta_i-\theta}{\theta_i}\right)\end{bsmallmatrix}\!.
    \end{equation}
    Using Taylor's expansion of $\sinh$, equations~\eqref{eq:assymp1},\eqref{eq:assymp2} lead to the following $\theta$-uniform asymptotic behavior 
    \begin{equation}
        F_k = \begin{bsmallmatrix} 0\\ \mathrm{sin}\left(k\pi\frac{\theta_i-\theta}{\theta_i}\right)\end{bsmallmatrix} + \underset{k\to\infty}{\mathcal{O}}\left(\frac{1}{k}\right).
    %\end{aligned}    
    \end{equation}
    Consider now the sequence $\mathcal{S}_E:=\{\mathbf{e}_1^{n_x},\cdots,\mathbf{e}_{n_x}^{n_x},E_k\}_{k\in\mathbb{N}}$ where
    \begin{equation}
    E_k := \begin{bsmallmatrix}0\\\sin\left(k\pi\frac{\theta_i-\theta}{\theta_i})\right)\end{bsmallmatrix}.
    %\left\{
    %    \begin{array}{ll}
    %    \begin{bsmallmatrix} \mathbf{e}_{-k+1}^{n_x}\\0\end{bsmallmatrix} & \text{for } k\in\{-n_x+1,\dots,0\},\\
    %    \begin{bsmallmatrix}0\\\sin\left(k\pi\frac{\theta_i-\theta}{\theta_i})\right)\end{bsmallmatrix}& \text{for } k>0.
    %    \end{array}
    %\right.
    \end{equation}
    This canonical sequence $\mathcal{S}_E$ forms a complete orthogonal basis of~$\mathcal{H}$. 
    By comparing the sequences $\mathcal{S}_E$ and $\mathcal{S}_F$ on $\mathcal{H}$, we obtain
    \begin{equation}
        \sum_{k=n}^{\infty}\lVert E_k-F_k \rVert^2 \leq \sum_{k=n}^{\infty} \underset{k\to\infty}{\mathcal{O}}\left(\frac{1}{k^2}\right) < \infty.
    \end{equation}
    Modified Bari's theorem in~\cite[Thm~6.3]{Guo2001} concludes the proof.
\end{proof}

As a consequence of Lemma~\ref{lem:riesz}, system~\eqref{eq:system} with an initial condition $\begin{bsmallmatrix}x(0)\\z(0)\end{bsmallmatrix}$ in $\mathcal{H}$ is well-posed. More precisely, the operator $\mathcal{A}$ in~\eqref{eq:A} generates a holomorphic semigroup for $\mathcal{H}$~\cite[Theorem 3.2.14]{Curtain2020}.

\subsection{Problem statement}

Recall first the definition of asymptotic stability~\cite{Karafyllis2011}.

\begin{defn}\label{def:stab}
    The equilibirum of system~\eqref{eq:system} is globally asymptotically stable if the two following properties hold.
    \begin{itemize}
    \item[(i)] Lyapunov stability: For all $\varepsilon>0$, there exists $\delta>0$ such that,  for any $\left\lVert\begin{bsmallmatrix}x_0\\z_0\end{bsmallmatrix}\right\rVert\leq \delta$, we have $\left\lVert\begin{bsmallmatrix}x(t)\\z(t)\end{bsmallmatrix}\right\rVert \leq \varepsilon$, $\forall t\geq 0$.
    \item[(ii)] Global attractivity: For any $\begin{bsmallmatrix}x_0\\z_0\end{bsmallmatrix}\in\mathcal{H}$, the solution $\begin{bsmallmatrix}x(t)\\z(t)\end{bsmallmatrix}\in\mathcal{H}$ converges to the origin as $t\to\infty$.
    \end{itemize} 
    %System~\eqref{eq:system} is unstable if there exists an initial condition $\begin{bsmallmatrix}x_0\\z_0\end{bsmallmatrix}\in\mathcal{H}$ such that the solution of~\eqref{eq:system} does not tend to zero as $t\to\infty$.
\end{defn}

In \cite[Chapter~8]{Karafyllis2019}, stability properties are discussed for reaction coefficient $\lambda<0$ and state matrix $A$ Hurwitz. In~\cite{Bajodek2022,Baudouin2019}, cases of unstable PDE and unstable ODE are respectively studied and present numerical formulations to verify sufficient conditions of stability.%. Nevertheless, the numerical conditions presented in these references are only sufficient and restrictive.

In the rest of the paper, we focus on necessary conditions of stability and aim at obtaining numerical conditions of instability. Using spectral and temporal approaches, respectively in Sections~III and IV, we determine when the assumptions (i)-(ii) in Definition~\ref{def:stab} do not hold. In other words, we wish to solve the following problem.
% for which combinations of parameters $A,B,C,\nu,\lambda,\theta_i,\theta_o$ system~\eqref{eq:system} is not globally asymptotically stable?
%How to determine instability regions with respect to the parameters $A,B,C,\nu,\lambda,\theta_i,\theta_o$?

\begin{pb}
Determine sets on the space of parameters $(A,B,C,\nu,\lambda,\theta_i,\theta_o)\in\mathbb{R}^{n_x\times n_x}\times \mathbb{R}^{n_x\times 1}\times \mathbb{R}^{1\times n_x} \times \mathbb{R}^4$ in which system~\eqref{eq:system} is not globally asymptotically stable.
\end{pb}

\section{Spectral analysis}
\label{sec:spectral}

Stability properties are often characterized by the poles' location. Root locus analysis enables to study the characteristic roots in terms of variation of the system parameters. For transport-ODE interconnections (time-delay systems), the literature abounds~\cite{Niculescu2010,Sipahi2011} and numerical solutions are proposed~\cite{Breda2014,Fioravanti2012}.
Here, we follow a similar path to study reaction-diffusion-ODE systems. %What about an extension to ODE-reaction-diffusion systems?
%Fioravanti2012

\subsection{Spectral condition}

From~\cite[Theorem 3.2.8]{Curtain2020}, the following general theorem is stated. 

\begin{thm}\label{thm:spectral}
    System~\eqref{eq:system} is globally asymptotically stable if and only if all the solutions of~\eqref{eq:charac} have a strictly negative real part.
\end{thm}
%a characteristic root~$s_k$, solution of
\begin{proof}
We denote by $\{\mathcal{F}_k\}_{k\in\mathbb{N}}$ and $\{\mathcal{F}_k^\ast\}_{k\in\mathbb{N}}$ the generalized characteristic functions associated to the characteristic roots $\{\sigma_k\}_{k\in\mathbb{N}}$ of $\mathcal{A}$ and $\{\sigma_k^\ast\}_{k\in\mathbb{N}}$ of  $\mathcal{A}^\ast$, such that $\{\mathcal{F}_k\}_{k\in\mathbb{N}}$ and $\{\mathcal{F}_k^\ast\}_{k\in\mathbb{N}}$ are biorthogonal. We also denote by $\delta_k$ the dimension of the $k$-th generalized characteristic space, which verify $\delta_k=1$, for any $k\geq n$ according to the proof of Lemma~\ref{lem:riesz}. The semigroup $\mathcal{T}(t)$ generated by the operator $\mathcal{A}$ is then given by
\begin{equation}\label{eq:T}
\begin{aligned}
    \mathcal{T}(t)\begin{bsmallmatrix}x_0\\z_0\end{bsmallmatrix} =& \sum_{k=1}^{n-1}\sum_{d=0}^{\delta_k}\alpha_{k,d}t^de^{s_k t}
\scal{\begin{bsmallmatrix}x_0\\z_0\end{bsmallmatrix}}{\mathcal{F}_k^\ast}\mathcal{F}_k\\
         &+ \sum_{k=n}^{\infty}e^{s_k t}
\scal{\begin{bsmallmatrix}x_0\\z_0\end{bsmallmatrix}}{\mathcal{F}_k^\ast}\mathcal{F}_k, \quad t\geq 0,
\end{aligned}
\end{equation}
for some scalars $\alpha_{k,d}$ and for any $\begin{bsmallmatrix}x_0\\z_0\end{bsmallmatrix}$ in $\mathcal{H}$.
The trajectories given by $\begin{bsmallmatrix}x(t)\\z(t)\end{bsmallmatrix}=\mathcal{T}(t)\begin{bsmallmatrix}x_0\\z_0\end{bsmallmatrix}$ satisfy both items~(i) and (ii) in Definition~\ref{def:stab} if and only if $\underset{k\in\mathbb{N}}{\mathrm{sup}}\mathrm{Re}(s_k)<0$, which concludes the proof.
\end{proof}

\subsection{Modified spectral condition}

With the change of variable $\sigma=\sqrt{\frac{s-\lambda}{\nu}}\theta_i$, the characteristic equation~\eqref{eq:charac} rewrites as
\begin{equation}\label{eq:tdscharac}
    \mathrm{det}((\nu\sigma^2+\lambda)I_{n_x}-A-B\bar{H}(\sigma)C) = 0,
\end{equation}
where the function $\bar{H}$ is given by
\begin{equation}\label{eq:Hsigma}
    \bar{H}(\sigma) := \frac{\sigma}{\theta_i}\begin{bmatrix}e^{\sigma\theta_o/\theta_i}\\-e^{-\sigma\theta_o/\theta_i}\end{bmatrix}^{\top}\begin{bmatrix} 1 & 1 \\ e^{\sigma} & e^{-\sigma} \end{bmatrix}^{-1} \begin{bmatrix}1\\0\end{bmatrix}.
\end{equation}

\begin{cor}\label{cor:spectral}
System~\eqref{eq:system} is globally asymptotically stable if and only if all the solutions $\sigma$ of~\eqref{eq:tdscharac} satisfy 
\begin{equation}\label{eq:spectralcond}
\left(\mathrm{Re}(\sigma/\theta_i)\right)^2\!-\!\left(\mathrm{Im}(\sigma/\theta_i)\right)^2 < - \frac{\lambda}{\nu}.
\end{equation}
\end{cor}
\begin{proof}
    Assume that $\sigma=\mathrm{Re}(\sigma)+\mathbf{i}\mathrm{Im}(\sigma)$ is solution of~\eqref{eq:tdscharac}. Then, $\frac{s-\lambda}{\nu}=\mathrm{Re}(\sigma/\theta_i)^2-\mathrm{Im}(\sigma/\theta_i)^2 + 2\mathbf{i}\mathrm{Re}(\sigma/\theta_i)\mathrm{Im}(\sigma/\theta_i)$
    is solution of~\eqref{eq:charac}. From Theorem~\ref{thm:spectral}, $\mathrm{Re}(s)<0$ is a necessary and sufficient condition of asymptotic stability. Since $\nu>0$, it can therefore be rewritten equivalently as in~\eqref{eq:spectralcond}.
\end{proof}

\begin{figure}[!t]
    \centering
    \begin{subfigure}{0.45\linewidth}
    \begin{tikzpicture}[line cap=round,line join=round,>=triangle 45,x=0.4cm,y=0.4cm]
    \begin{axis}[
    x=0.4cm,y=0.4cm,
    axis lines=none,
    ymajorgrids=false,
    xmajorgrids=false,
    xmin=-4,
    xmax=5.5,
    ymin=-2,
    ymax=3,]
    \clip(-4,-3.5) rectangle (5.5,3.5);
    \draw [->,line width=1pt] (-3.5,0)-- (3.5,0);
    \draw [->,line width=1pt] (0,-2)-- (0,2);
    \filldraw [fill=green!80!black,nearly transparent] (-3.5,-2) rectangle (0,2);
    \filldraw [fill=red!80!black,nearly transparent] (0,-2) rectangle (3.5,2);
    \draw (-1.5,0.5) node {Stable};
    \draw (-1.5,-0.5) node {roots};
    \draw (1.5,0.5) node {Unstable};
    \draw (1.5,-0.5) node {roots};
    \draw (4.5,-0.2) node {$\mathrm{Re}(s)$};
    \draw (0.5,2.5) node {$\mathrm{Im}(s)$};
    \end{axis}
    \end{tikzpicture}
    \subcaption{Stability condition \textit{before} the change of variable.}
    \end{subfigure}
    \hfill
    \begin{subfigure}{0.45\linewidth}
    \begin{tikzpicture}[line cap=round,line join=round,>=triangle 45,x=0.4cm,y=0.4cm]
    \begin{axis}[
    x=0.4cm,y=0.4cm,
    axis lines=none,
    ymajorgrids=false,
    xmajorgrids=false,
    xmin=-4,
    xmax=5.5,
    ymin=-2,
    ymax=3,]
    \clip(-4,-3.5) rectangle (5.5,3.5);
    \draw [->,line width=1pt] (-3.5,0)-- (3.5,0);
    \draw [->,line width=1pt] (0,-2)-- (0,2);
    \draw [fill=red!80!black,nearly transparent] (0,0) -- plot [domain=0:3.5] (\x,{min(\x,2)}) -- (3.5,0) -- cycle;
    \draw [fill=red!80!black,nearly transparent] (0,0) -- plot [domain=0:3.5] (\x,{-min(\x,2)}) -- (3.5,0) -- cycle;
\draw [fill=red!80!black,nearly transparent] (0,0) -- plot [domain=0:-3.5] (\x,-{min(-\x,2)}) -- (-3.5,0) -- cycle;
    \draw [fill=red!80!black,nearly transparent] (0,0) -- plot [domain=0:-3.5] (\x,{min(-\x,2)}) -- (-3.5,0) -- cycle;
    \draw [fill=green!80!black,nearly transparent] plot [domain=-2:2] (\x,{abs(\x)}) -- plot [domain=-2:2] (\x,2) -- cycle;
    \draw [fill=green!80!black,nearly transparent] plot [domain=-2:2] (\x,-{abs(\x)}) -- plot [domain=-2:2] (\x,-2) -- cycle;
    \draw (4.5,-0.2) node {$\mathrm{Re}(\sigma)$};
    \draw (0.5,2.5) node {$\mathrm{Im}(\sigma)$};
    \end{axis}
    \end{tikzpicture}
    \subcaption{Stability condition \textit{after} the change of variable, for $\lambda = 0$.}
    \end{subfigure}

    \begin{subfigure}{0.45\linewidth}
    \begin{tikzpicture}[line cap=round,line join=round,>=triangle 45,x=0.4cm,y=0.4cm]
    \begin{axis}[
    x=0.4cm,y=0.4cm,
    axis lines=none,
    ymajorgrids=false,
    xmajorgrids=false,
    xmin=-4,
    xmax=5.5,
    ymin=-2,
    ymax=3,]
    \clip(-4,-3.5) rectangle (5.5,3.5);
    \draw [->,line width=1pt] (-3.5,0)-- (3.5,0);
    \draw [->,line width=1pt] (0,-2)-- (0,2);
    \draw [fill=red!80!black,nearly transparent] (0,0) -- (1,0) -- plot [domain=1:3.5] (\x,{min({sqrt(\x*\x-1)},2)}) -- (3.5,0) -- cycle;
    \draw [fill=red!80!black,nearly transparent] (0,0) -- (1,0) -- plot [domain=1:3.5] (\x,{-min({sqrt(\x*\x-1)},2)}) -- (3.5,0) -- cycle;
\draw [fill=red!80!black,nearly transparent] (0,0) -- (-1,0) -- plot [domain=-1:-3.5] (\x,{min({sqrt(\x*\x-1)},2)}) -- (-3.5,0) -- cycle;
    \draw [fill=red!80!black,nearly transparent] (0,0) -- (-1,0) -- plot [domain=-1:-3.5] (\x,{-min({sqrt(\x*\x-1)},2)}) -- (-3.5,0) -- cycle;
    \draw [fill=green!80!black,nearly transparent] plot [domain=-{sqrt(5)}:{sqrt(5)}] (\x,{sqrt(max(0,(\x*\x-1)))}) -- plot [domain=-{sqrt(5)}:{sqrt(5)}] (\x,2) -- cycle;
    \draw [fill=green!80!black,nearly transparent] plot [domain=-{sqrt(5)}:{sqrt(5)}] (\x,-{sqrt(max(0,(\x*\x-1)))}) -- plot [domain=-{sqrt(5)}:{sqrt(5)}] (\x,-2) -- cycle;
    \draw (4.5,-0.2) node {$\mathrm{Re}(\sigma)$};
    \draw (0.5,2.5) node {$\mathrm{Im}(\sigma)$};
    \end{axis}
    \end{tikzpicture}
    \subcaption{Stability condition \textit{after} the change of variable, for $\lambda < 0$.}
    \end{subfigure}
    \hfill
    \begin{subfigure}{0.45\linewidth}
    \begin{tikzpicture}[line cap=round,line join=round,>=triangle 45,x=0.4cm,y=0.4cm]
    \begin{axis}[
    x=0.4cm,y=0.4cm,
    axis lines=none,
    ymajorgrids=false,
    xmajorgrids=false,
    xmin=-4,
    xmax=5.5,
    ymin=-2,
    ymax=3,]
    \clip(-4,-3.5) rectangle (5.5,3.5);
    \draw [->,line width=1pt] (-3.5,0)-- (3.5,0);
    \draw [->,line width=1pt] (0,-2)-- (0,2);
    \draw [fill=red!80!black,nearly transparent] (0,0) -- (0,1) -- plot [domain=0:3.5] (\x,{min({sqrt(\x*\x+1)},2)}) -- (3.5,0) -- cycle;
    \draw [fill=red!80!black,nearly transparent] (0,0) -- (0,-1) -- plot [domain=0:3.5] (\x,-{min({sqrt(\x*\x+1)},2)}) -- (3.5,0) -- cycle;
    \draw [fill=red!80!black,nearly transparent] (0,0) -- (0,1) -- plot [domain=0:-3.5] (\x,{min({sqrt(\x*\x+1)},2)}) -- (-3.5,0) -- cycle;
    \draw [fill=red!80!black,nearly transparent] (0,0) -- (0,-1) -- plot [domain=0:-3.5] (\x,-{min({sqrt(\x*\x+1)},2)}) -- (-3.5,0) -- cycle;
    \draw [fill=green!80!black,nearly transparent] plot [domain=-{sqrt(3)}:{sqrt(3)}] (\x,{sqrt(\x*\x+1)}) -- plot [domain=-{sqrt(3)}:{sqrt(3)}] (\x,2) -- cycle;
    \draw [fill=green!80!black,nearly transparent] plot [domain=-{sqrt(3)}:{sqrt(3)}] (\x,-{sqrt(\x*\x+1)}) -- plot [domain=-{sqrt(3)}:{sqrt(3)}] (\x,-2) -- cycle;
    \draw (4.5,-0.2) node {$\mathrm{Re}(\sigma)$};
    \draw (0.5,2.5) node {$\mathrm{Im}(\sigma)$};
    \end{axis}
    \end{tikzpicture}
    \subcaption{Stability condition \textit{after} the change of variable, for $\lambda > 0$.}
    \end{subfigure}
    
    \caption{Illustration of Theorem~\ref{thm:spectral} and Corollary~\ref{cor:spectral}.}
    \label{fig:spectralthm}
\end{figure}
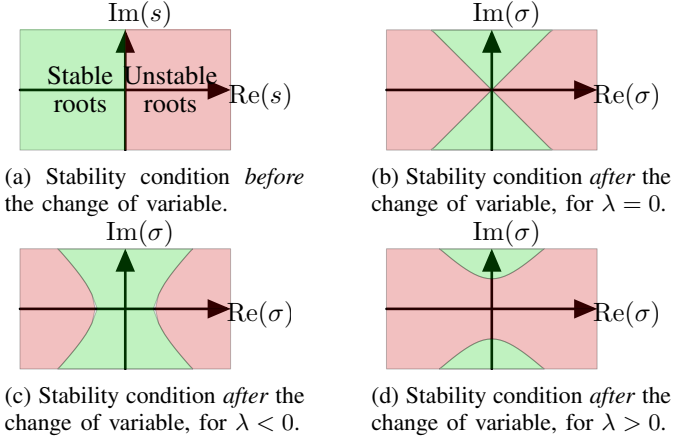

Even though Theorem~\ref{thm:spectral} and Corollary~\ref{cor:spectral} present necessary and sufficient conditions for global asymptotic stability of system~\eqref{eq:system} (see Fig.~\ref{fig:spectralthm}), they require the knowledge of the entire point spectrum of $\mathcal{A}$, which is difficult to obtain in general. On the other hand, it is possible to compute a subset of the point spectrum. Therefore, Theorem~\ref{thm:spectral} or Corollary~\ref{cor:spectral} will only be used as sufficient instability conditions, that is, if at least one characteristic root has a positive real part, then the origin of~\eqref{eq:system} is unstable.

\begin{rem}
The conditions in Theorem~\ref{thm:spectral} or Corollary~\ref{cor:spectral} are similar to the ones detailed in~\cite{Bonnet2000} for fractional differential systems.
\end{rem}
\begin{rem}
In practice, note that the solutions of~\eqref{eq:tdscharac} are related to the eigenvalue problem of a neutral time-delay system. The characteristic equation~\eqref{eq:tdscharac} is then easier to study than~\eqref{eq:charac}. Indeed, matrix pencils~\cite{Niculescu2005} or discretization~\cite{Breda2004} techniques can be used to plot the root locus graphic.
\end{rem}

\section{Lyapunov analysis}
\label{sec:lyapunov}

The stability analysis of linear infinite-dimensional systems can also be studied by Lyapunov methods~\cite{Datko1968,Datko1970}. % (time-delay systems)
For transport-ODE interconnections, a necessary and sufficient condition of stability is based on the positivity of a converse Lyapunov operator written in terms of the Lyapunov delay matrix~\cite{Kharitonov2006}. Approximate solutions and numerical tests using a finite number of parameters are then proposed~\cite{Gomez2021,Mondie2022}.
This section presents a similar formulation to obtain instability certificates for reaction-diffusion-ODE systems. %What about an extension to ODE-reaction-diffusion systems?

\subsection{Lyapunov functional}

For any $\begin{bsmallmatrix}x\\z\end{bsmallmatrix}\in\mathcal{H}$, let us introduce the quadratic functional % weighted 
\begin{equation}\label{eq:lyap}
\begin{aligned}
    \mathcal{V}(\begin{bsmallmatrix}x\\z\end{bsmallmatrix}) &= 
    x^\top P x + 2\int_0^{\theta_i} x^\top Q(\theta)z(\theta)\mathrm{d}\theta\\
    &\quad+ \int_0^{\theta_i} \int_0^{\theta_i} z^\top(\theta_1)T(\theta_1,\theta_2)z(\theta_2)\mathrm{d}\theta_1\mathrm{d}\theta_2.
\end{aligned}
\end{equation}
with $P\in\mathbb{S}^{n_x}_+$, $T(\theta_1,\theta_2)=T^{\top}(\theta_2,\theta_1)$ and bounded functions $Q\in L^2(0,\theta_i)^{n_x}$ and $T\in L^2((0,\theta_i)\times(0,\theta_i))$.

For any $\begin{bsmallmatrix}x\\z\end{bsmallmatrix}\in\mathcal{D}$, its time derivative along the trajectories of system~\eqref{eq:system} is given by
\begin{align}\label{eq:lyapder}
    &\frac{1}{2}\dot{\mathcal{V}}(\begin{bsmallmatrix}x\\z\end{bsmallmatrix}) \!=\! x^\top\!P(Ax \!+\! B \partial_\theta z(\theta_o)) \!+\! \partial_{\theta}z^\top\!(\theta_o)\!\!\int_0^{\theta_i}\!\!\! B^\top Q(\theta) z(\theta)\mathrm{d}\theta \nonumber\\
    & + x^\top \int_0^{\theta_i} ( A^\top Q(\theta)z(\theta) \!+\! Q(\theta)(\nu\partial_{\theta\theta}+\lambda)z(\theta))\mathrm{d}\theta  \nonumber\\
    & + \int_0^{\theta_i}\int_0^{\theta_i}z^\top(\theta_1)T(\theta_1,\theta_2)(\nu\partial_{\theta_2\theta_2}+\lambda)z(\theta_2)\mathrm{d}\theta_1\mathrm{d}\theta_2.
\end{align}
This last expression will be used below in the proof of Theorem~\ref{thm:conv}.

\begin{thm}\label{thm:lyap}
System~\eqref{eq:system} is globally asymptotically stable if and only if there exist scalars $\gamma_1,\gamma_2>0$ and a Lyapunov functional~$\mathcal{V}$ of the form~\eqref{eq:lyap} satisfying
\begin{subequations}\label{eq:ineq}
\begin{empheq}[left=\empheqlbrace]{align}
    \gamma_1\left\lVert \begin{bsmallmatrix}x\\z\end{bsmallmatrix} \right\rVert^2 &\leq \mathcal{V}(\begin{bsmallmatrix}x\\z\end{bsmallmatrix}),&&\forall \begin{bsmallmatrix}x\\z\end{bsmallmatrix}\in\mathcal{H},\label{eq:ineq1}\\
    \dot{\mathcal{V}}(\begin{bsmallmatrix}x\\z\end{bsmallmatrix}) &\leq -\gamma_2\left\lVert \begin{bsmallmatrix}x\\z\end{bsmallmatrix}\right\rVert^2,&&\forall \begin{bsmallmatrix}x\\z\end{bsmallmatrix}\in\mathcal{D}.\label{eq:ineq2}
\end{empheq}
\end{subequations}
\end{thm}

\begin{proof}
The sufficiency is verified by application of the Lyapunov theorem in the Hilbert space $\mathcal{H}$ as in~\cite[Corollary~2]{Datko1968}. The necessity follows from~\cite[Theorem~1]{Datko1968}. Indeed, assuming that the system is asymptotically stable, the converse Lyapunov operator $\mathcal{P}$ solution of the Lyapunov equation $\mathcal{A}^\ast \mathcal{P} + \mathcal{P}\mathcal{A} = -\mathcal{I}$ is expressed as
\begin{equation}\label{eq:Pconv}
\begin{aligned}
    %\scal{\begin{bsmallmatrix}x_0\\z_0\end{bsmallmatrix}}
    &\mathcal{P}\begin{bsmallmatrix}x_0\\z_0\end{bsmallmatrix} = -\int_0^{\infty} \mathcal{T}^\ast(t)\mathcal{T}(t)\begin{bsmallmatrix}x_0\\z_0\end{bsmallmatrix}\mathrm{d}t,
\end{aligned}
\end{equation}
where $\mathcal{T}$ is the semigroup generated by $\mathcal{A}$.
Using the expression~\eqref{eq:T} of $\mathcal{T}$ leads to
\begin{equation}\label{eq:Pconv2}
\mathcal{P}\begin{bsmallmatrix}x_0\\z_0\end{bsmallmatrix} = \sum_{k=1}^{\infty}\sum_{k^\prime=1}^{\infty}\mathcal{I}_{k,k^\prime}\scal{\begin{bsmallmatrix}x_0\\z_0\end{bsmallmatrix}}{\mathcal{F}_k^\ast}\scal{\mathcal{F}_k}{\mathcal{F}_{k^\prime}^\ast}\mathcal{F}_{k^\prime},
\end{equation}
 where the integral $\mathcal{I}_{k,k^\prime}$ is given by
\begin{equation*}
    \mathcal{I}_{k,k^\prime} = -\int_0^{\infty} \beta_{k,k^\prime}(t) e^{(s_k+s_{k^\prime}^\ast)t} \mathrm{d}t,
\end{equation*}
with
\begin{equation*}
    \beta_{k,k^\prime}(t) = \left\{ 
    \begin{array}{ll}
        \displaystyle \!\!\sum_{d=0}^{\delta_k}\!\sum_{d^\prime=0}^{\delta_{k^\prime}} \!\!\left(\!\alpha_{k,d}t^d\right)\!\!\left(\!\alpha_{k^\prime,d^\prime}^\ast t^{d^\prime}\!\right)\!\! &\text{ if } k<n \text{ or } k^\prime <n,\\
        %\displaystyle \sum_{d=0}^{\delta_k}\left(\alpha_{k,d}t^d\right) &\text{ if } k<n,\, k^\prime \geq n,\\
        %\displaystyle \sum_{d^\prime=0}^{\delta_{k^\prime}} \left(\alpha_{k^\prime,d^\prime}^\ast t^{d^\prime}\right) &\text{ if } k\geq n,\, k^\prime <n,\\
        \displaystyle 1 &\text{ if } k\geq n \text{ and } k^\prime \geq n,
    \end{array}
    \right.
\end{equation*}
and with scalars $\alpha_{k,d}$ introduced in~\eqref{eq:T}.
The biorthogonality of the generalized characteristic functions $\{\mathcal{F}_k\}_{k\in\mathbb{N}}$ of $\mathcal{A}$, $\{\mathcal{F}_k^\ast\}_{k\in\mathbb{N}}$ of $\mathcal{A}^\ast$ simplifies the expression 
\begin{equation}
    \mathcal{P}\begin{bsmallmatrix}x_0\\z_0\end{bsmallmatrix} = \sum_{k=1}^{\infty}\mathcal{I}_{k,k}\scal{\begin{bsmallmatrix}x_0\\z_0\end{bsmallmatrix}}{\mathcal{F}_k^\ast}\mathcal{F}_k,
\end{equation}
where the integral is computed with integration by parts
\begin{equation*}
\mathcal{I}_{k,k} = \left\{ 
    \begin{array}{ll}
        \displaystyle\sum_{d=0}^{\delta_k}\sum_{d^\prime=0}^{\delta_{k}} \frac{(-1)^{d+d^\prime}\alpha_{k,d}\alpha_{k,d^\prime}^\ast}{(2\mathrm{Re}(s_k))^{d+d^\prime +1}} &\text{ if } k<n,\\
        \displaystyle-\frac{1}{2\mathrm{Re}(s_k)} &\text{ otherwise.}\\
    \end{array}
    \right.
\end{equation*}
Therefore, the converse Lyapunov functional 
\begin{equation}
    \mathcal{V}(\begin{bsmallmatrix}x_0\\z_0\end{bsmallmatrix})= \scal{\begin{bsmallmatrix}x_0\\z_0\end{bsmallmatrix}}{\mathcal{P}\begin{bsmallmatrix}x_0\\z_0\end{bsmallmatrix}}, \quad \forall \begin{bsmallmatrix}x_0\\z_0\end{bsmallmatrix}\in\mathcal{H},
\end{equation}
can be written in the form~\eqref{eq:lyap}. On one side,
\cite[Theorem~2]{Datko1968} ensures the boundedness of $\mathcal{P}$ and gives $\mathrm{Re}(s_k)<0$ for all $k\in\mathbb{N}$ which guarantees the positivity of $\mathcal{P}$ and the validity of~\eqref{eq:ineq1}. On the other side, along the trajectories of system~\eqref{eq:system}, the inequality~\eqref{eq:ineq2} holds. This completes the proof. %verify $\dot{\mathcal{V}}(\begin{bsmallmatrix}x\\z\end{bsmallmatrix})=-\lVert\begin{bsmallmatrix}x\\z\end{bsmallmatrix}\rVert^2$ in $\mathcal{D}$.
% with appropriate bounded functions $(P,Q,T)$ 
\end{proof}
%For more details, one can refer to~\cite{Datko1968,Datko1970}.

\subsection{Direct Lyapunov condition}

For a given integer $n\in\mathbb{N}$, consider a set of linearly independent functions~$\{\varphi_{k}\}_{k\in\{0,\dots,n-1\}}$ in $L^2(0,\theta_i)$ which is
\begin{itemize}
\item[(P1)] orthonormal with the usual $L^2(0,\theta_i)$ scalar product,
\item[(P2)] closed under differentiation.
\end{itemize}
%\begin{rem}
For instance, sets of Legendre polynomials $\{l_k\}_{k\in\mathbb{N}}$~or trigonometric functions $\{\cos(k\pi\frac{\theta}{\theta_i}),\sin(k\pi\frac{\theta}{\theta_i})\}_{k\in\mathbb{N}}$ satisfy properties (P1)-(P2).
%\end{rem}

In this subsection, we intend to perform a projection of the infinite-dimensional spaces $\mathcal{H}$ and $\mathcal{D}$~\eqref{eq:domains} into the finite-dimensional subspace % \mathcal{S}
spanned by the sequence~$\mathcal{S}_n:=\{\mathbf{e}_1^{n_x},\cdots,\mathbf{e}_{n_x}^{n_x},\varphi_k\}_{k\in\{0,\dots,n-1\}}$, that are
\begin{equation*}
H_n := \mathrm{span}(\mathcal{S}_n)\subset \mathcal{H},\quad
D_n := H_n \cap \mathcal{D} \subset \mathcal{D}.
\end{equation*}

For any $(x,z)\in H_n$, using the orthonormal property~(P1), the state $z$ can be decomposed in a finite number of terms as follows
\begin{equation}
    z(\theta) \!=\! \sum_{k=0}^{n-1}\varphi_k(\theta)\!\!\int_0^{\theta_i}\!\!\! \varphi_k(\tau)z(\tau)\mathrm{d}\tau \!=\! \Phi_n^\top(\theta) \!\!\int_0^{\theta_i}\!\!\! \Phi_n(\tau)z(\tau)\mathrm{d}\tau,
\end{equation}
with 
$\Phi_n=\begin{bsmallmatrix}\varphi_0&\cdots&\varphi_{n-1}\end{bsmallmatrix}^\top\in \mathbb{R}^{n\times 1}$,
the Lyapunov functional in~\eqref{eq:lyap} is equal to
\begin{equation}
    \mathcal{V}(\begin{bsmallmatrix}x\\z\end{bsmallmatrix}) = \begin{bsmallmatrix}x\\\int_0^{\theta_i} \Phi_n(\tau)z(\tau)\mathrm{d}\tau\end{bsmallmatrix}^\top \Psi_n^+ \begin{bsmallmatrix}x\\\int_0^{\theta_i} \Phi_n(\tau)z(\tau)\mathrm{d}\tau\end{bsmallmatrix},
\end{equation}
with matrices
\begin{equation}\label{eq:psin1}
\begin{aligned}
    \Psi_n^+ &= \begin{bmatrix}P&Q_n\\Q_n^\top&T_n\end{bmatrix} \in \mathbb{R}^{(n_x+n)\times (n_x +n)},\quad P \in \mathbb{S}^{n_x}_+,\\
    Q_n &= \int_0^{\theta_i}Q(\theta)\Phi_n^\top(\theta)\mathrm{d}\theta\in\mathbb{R}^{n_x\times n},\\
    T_n &= \int_0^{\theta_i}\int_0^{\theta_i}\Phi_n(\theta_1)T( \theta_1,\theta_2)\Phi_n^\top(\theta_2)\mathrm{d}\theta_1\theta_2 \in\mathbb{S}^{n}_+.
\end{aligned}
\end{equation}
%satisfying $P=P^\top$ and $T_n=T_n^\top$.

For any $(x,z)\in H_n$, using the differentiation property~(P2),
\begin{equation}
    \Phi_n^\prime(\theta) = \Delta_n \Phi_n(\theta),\qquad\forall \theta\in[0,\theta_i],
\end{equation}
for some differentiation matrix $\Delta_n\in\mathbb{R}^{n\times n}$ with known coefficients, the Lyapunov functional's derivative~\eqref{eq:lyapder} is equal to
\begin{equation}
    \dot{\mathcal{V}}(\begin{bsmallmatrix}x\\z\end{bsmallmatrix}) = \begin{bsmallmatrix}x\\\int_0^{\theta_i}\Phi_n(\tau)z(\tau)\mathrm{d}\tau\end{bsmallmatrix}^\top \Psi_n^- \begin{bsmallmatrix}x\\\int_0^{\theta_i}\Phi_n(\tau)z(\tau)\mathrm{d}\tau\end{bsmallmatrix},
\end{equation}
with matrices
\begin{equation}\label{eq:psin2}
\begin{aligned}
    \Psi_n^-&= \mathrm{He}\left(\begin{bmatrix}\Psi_{xx}&\Psi_{xz}\\0&\Psi_{zz}\end{bmatrix}\right)\in \mathbb{R}^{(n_x+n)\times (n_x+n)},\\
    \Psi_{xx} &= PA \in\mathbb{R}^{n_x\times n_x},\\
    \Psi_{xz} &= PB\Phi_n^\top(\theta_o)\Delta_n^\top\!+\!(\lambda I_{n_x}\!+\!A^\top) Q_n\!+\!\nu Q_n(\Delta_n^{\top})^2 \in\mathbb{R}^{n_x\times n},\\
    \Psi_{zz} &= \Delta_n\Phi_n(\theta_o)B^\top Q_n + T_n(\lambda I_{n} \!+\! \nu (\Delta_n^{\top})^2) \in\mathbb{R}^{n\times n}.
\end{aligned}
\end{equation}

Moreover, considering $(x,z)\in D_n$, the expression of $\dot{\mathcal{V}}$ will be restricted to $\mathcal{D}$ in~\eqref{eq:domains}. In other words, the vector $\begin{bsmallmatrix}x\\\int_0^{\theta_i} \Phi_n(\tau)z(\tau)\mathrm{d}\tau\end{bsmallmatrix}$ need to satisfy 
\begin{equation}\label{eq:constraint}
\begin{bsmallmatrix}C&-\Phi_n^\top(0)\\0&-\Phi_n^\top(\theta_i)\end{bsmallmatrix}\begin{bsmallmatrix}x\\\int_0^{\theta_i} \Phi_n(\tau)z(\tau)\mathrm{d}\tau\end{bsmallmatrix}=0.
\end{equation}
This constraint~\eqref{eq:constraint} can be seen as a projection of the finite-dimensional state $\xi_n=\begin{bsmallmatrix}x\\\int_0^{\theta_i} \Phi_n(\tau)z(\tau)\mathrm{d}\tau\end{bsmallmatrix}$ on the matrix kernel
\begin{equation}
    \Pi_n = \mathrm{ker}\begin{bsmallmatrix}C&-\Phi_n^\top(0)\\0&-\Phi_n^\top(\theta_i)\end{bsmallmatrix}.
\end{equation}

Matrices $\Psi_n^{+}$ and $\Psi_n^{-}$ are expressed with a finite number of parameters via the triplet of matrices $(P,Q_n,T_n)$. We can then deduce an instability criterion in the form of a semidefinite programming test based on the feasibility of two affine matrix inequalities.

\begin{thm}\label{thm2}
For a given $n\in\mathbb{N}$, if there is no triplet of matrices $(P,Q_n,T_n)\in\mathbb{S}^{n_x}_+\times\mathbb{R}^{n_x\times n}\times\mathbb{S}^{n}_+$ such that $\Psi_n^+\succ 0$ and $\Pi_n^\top\Psi_n^- \Pi_n\prec 0$, then system~\eqref{eq:system} is unstable.
\end{thm}
\begin{proof}
Assume that system~\eqref{eq:system} is globally asymptotically stable. Then, there exist scalars $\gamma_1,\gamma_2>0$ and a triplet of functions $(P,Q,T)$ such that the Lyapunov functional $\mathcal{V}$ given by~\eqref{eq:lyap} satisfy inequalities~\eqref{eq:ineq} that are recalled below
\begin{equation*}
    \begin{aligned}
        \gamma_1\left\lVert \begin{bsmallmatrix}x\\z\end{bsmallmatrix} \right\rVert^2&\leq\mathcal{V}(\begin{bsmallmatrix}x\\z\end{bsmallmatrix}),&&\forall \begin{bsmallmatrix}x\\z\end{bsmallmatrix}\in\mathcal{H},\\
        \dot{\mathcal{V}}(\begin{bsmallmatrix}x\\z\end{bsmallmatrix})&\leq-\gamma_2\left\lVert \begin{bsmallmatrix}x\\z\end{bsmallmatrix} \right\rVert^2,&&\forall \begin{bsmallmatrix}x\\z\end{bsmallmatrix}\in\mathcal{D}.
    \end{aligned}
\end{equation*}
In particular, these inequalities are also verified on the subsets $H_n\subset\mathcal{H}$ and $D_n\subset\mathcal{D}$, respectively. %Denoting $Q_n=\int_0^1 Q(\theta)\Phi_n^\top(\theta)\mathrm{d}\theta$ and $T_n=\int_0^1\int_0^1\Phi_n(\theta_1)T(\theta_1,\theta_2)\Phi_n^\top(\theta_2)\mathrm{d}\theta_1\mathrm{d}\theta_2)$, it gives from one side
It gives then a triplet of matrices $(P,Q_n,T_n)$ such that
\begin{equation}
\begin{aligned}
    \gamma_1|\xi_n|^2&\leq\xi_n^\top\Psi_n^-\xi_n,&& \forall \xi_n\in\mathbb{R}^{n_x+n},\\
    \xi_n^\top\Psi_n^-\xi_n &\leq -\gamma_2|\xi_n|^2,&& \forall \xi_n\in\left\{\begin{array}{l}\mathbb{R}^{n_x+n} \text{ such that } \\\begin{bsmallmatrix}C&-\Phi_n^\top(0)\\0&-\Phi_n^\top(\theta_i)\end{bsmallmatrix}\xi_n=0\end{array}\right\},
    %\begin{bsmallmatrix}P&\int_0^1 Q(\theta_2)\Phi_n^\top(\theta_2)\mathrm{d}\theta_2\\\int_0^1\Phi_n(\theta_1)Q^\top(\theta_1)\mathrm{d}\theta_1&\int_0^1\int_0^1\Phi_n(\theta_1)T(\theta_1,\theta_2)\Phi_n^\top(\theta_2)\mathrm{d}\theta_1\mathrm{d}\theta_2\end{bsmallmatrix}
\end{aligned}
\end{equation}
which means that $\Psi_n^+\succ0$ and $\Pi_n^\top \Psi_n^- \Pi_n\prec 0$. By contraposition, the theorem's statement holds.
\end{proof}

\begin{rem}
Note that the condition in Theorem~\ref{thm:lyap} gives an outer estimate of the feasible set in the space of parameters for the integral inequalities~\eqref{eq:ineq}. It naturally leads to a sufficient instability condition. Following~\cite{Fantuzzi2017} and using the add-on \textit{QUINOPT} (QUadratic INtegral OPTimisation) on \textit{Yalmip}, it might be possible to also have an inner estimate of the integral inequalities~\eqref{eq:ineq}. %Combining the inner and outer estimates, we could obtain necessary and sufficient stability conditions.
\end{rem}
%\begin{rem}
%Note also that this condition is written in a very compact way using the null space $\Pi_n$. For interpreting this result, the matrix $A_n=\Pi_n^\top\begin{bsmallmatrix}A&B\Phi_n^\top(\theta_o)\Delta_n^\top\\0&\lambda I_{n}+\nu\Delta_n^{\top 2}\end{bsmallmatrix}\Pi_n$, which is the approximated matrix obtained through the \textit{tau}-method, should be constructed and the matrix $\Pi_n^\top\Psi_n\Pi_n$ rewritten in terms of $A_n$. %note that the approximated matrix $A_n$ by \textit{tau}-method on the basis~$\{\varphi_k\}_{k\in\mathbb{N}}$ in~\cite{Bajodek2022} and $\Pi_n^\top\begin{bsmallmatrix}A&B\Phi_n^\top(\theta_o)\Delta_n^\top\\0&\lambda I_{n}+\nu\Delta_n^{\top 2}\end{bsmallmatrix}\Pi_n$ may be correlated. %They have the same eigenvalues.
%\end{rem}

\subsection{Converse Lyapunov condition}

In this subsection, we exploit the solution of the Lyapunov equation $\mathcal{A}^\ast \mathcal{P} + \mathcal{P}\mathcal{A} = - \mathcal{W}$, for an arbitrary diagonal and positive operator $\mathcal{W}=\begin{bsmallmatrix}1&0\\0&w\end{bsmallmatrix}$. Indeed, for $\theta_o=\theta_i$, the particular triplet~$(P,Q,T)$ associated to such an converse operator $\mathcal{P}$ satisfies the following equations (see Appendix~\ref{app2} for calculation details)
\begin{subequations}\label{eq:PQT}
\begin{empheq}[left=\empheqlbrace]{align}
    &\mathrm{He}(PA+\nu Q^\prime(0)C) = -I_{n_x},\label{eq:PQT1}\\
    &\nu Q^{\prime\prime}(\theta) \!+\! (A^\top\!+\!\lambda I_{n_x})Q(\theta) \!+\!\nu C^\top \partial_{\theta_1}T(0,\theta) \!=\! 0,\label{eq:PQT2}\\
    &Q(0)=0, \quad PB+\nu Q(\theta_i)=0,\label{eq:PQT3}\\
    &(\nu\partial_{\theta_1\theta_1}+\nu\partial_{\theta_2\theta_2}+2\lambda) T(\theta_1,\theta_2) = 0 ,\label{eq:PQT4}\\
    &T(\theta,0) = 0,  \quad \nu T(\theta_i,\theta)+B^\top Q(\theta)=0,\label{eq:PQT5}\\
    &(\partial_{\theta_1}-\partial_{\theta_2})T(\theta,\theta) =-\frac{w}{2}.\label{eq:PQT6}
\end{empheq}
\end{subequations}

According to~\cite[Theorem~1]{Datko1970}, the statement below holds.
\begin{thm}\label{thm3}
For $w>0$, consider the Lyapunov functional~$\mathcal{V}$ of the form~\eqref{eq:lyap} where $(P,Q,T)$ satisfy~\eqref{eq:PQT}. System~\eqref{eq:system} is globally asymptotically stable if and only if there exists a scalar $\gamma>0$ such that for all $\begin{bsmallmatrix}x\\z\end{bsmallmatrix}\in\mathcal{H}$ the following inequality holds
\begin{equation}\label{eq:ineq3}
\gamma\left\lVert \begin{bsmallmatrix}x\\z\end{bsmallmatrix} \right\rVert^2 \leq \mathcal{V}(\begin{bsmallmatrix}x\\z\end{bsmallmatrix}).
\end{equation}
\end{thm}
\begin{proof}
    It parallels the proof of Theorem~\ref{thm:lyap}. For more details, one can refer to~\cite{Datko1968,Datko1970}.
\end{proof}

Since inequality~\eqref{eq:ineq3} cannot be tested numerically, we propose an alternative as a sufficient condition of instability.

\begin{cor}\label{thm:conv}
For a given $n$ in $\mathbb{N}$, define matrix $\Psi_n^+$ by~\eqref{eq:psin1} where $(P,Q,T)$ satisfy~\eqref{eq:PQT} with $w>0$.
If $\Psi_n^+$ is not positive definite, then system~\eqref{eq:system} is unstable.
\end{cor}
\begin{proof}
Assume that $\Psi_n^+$ is not positive definite. Consequently, there exist a state $\xi_n=\begin{bsmallmatrix}X^+\\\zeta_n^+\end{bsmallmatrix}\in\mathbb{R}^{n_x+n}\backslash\{0\}$ such that $\xi_n^\top\Psi_n^+\xi_n\leq 0$. Then, by orthogonality of the functions $\Phi_n$, we have
\begin{equation*}
    \mathcal{V}(\begin{bsmallmatrix}x\\z\end{bsmallmatrix}) \leq 0, \text{ for } \begin{bsmallmatrix}x\\z\end{bsmallmatrix}:=\begin{bsmallmatrix}X^+\\\Phi_n^\top\zeta_n^+\end{bsmallmatrix}\in\mathcal{H}\backslash\{0\}.
\end{equation*}
The previous inequality shows that the converse Lyapunov functional is negative for a particular state in $H_n\subset \mathcal{H}$. By application of Theorem~\ref{thm3}, we conclude that system~\eqref{eq:system} is unstable.
\end{proof}

For the case $n=0$, the following corollary holds.
\begin{cor}\label{cor:convP}
    If the matrix $P$ solution of~\eqref{eq:PQT} with $w>0$ is not positive definite, then system~\eqref{eq:system} is unstable.
\end{cor}

\begin{rem}
Note that Corollaries~\ref{thm:conv} and \ref{cor:convP} extend Theorem~\ref{thm2} and can be used when the Lyapunov converse function is known. Obtaining such a function is far from trivial in the general case.
\end{rem}

\section{Numerical results}

\subsection{A scalar example}

The example below allows a complete and simple parametric study by considering $n_x=1$.
\begin{exmp}\label{exmp:1d}
    Consider system~\eqref{eq:system} with $A=a\in\mathbb{R}$, $B=b\in\mathbb{R}$, $C=1$, $\theta_i=\theta_o$, $\nu\in\mathbb{R}$ and $\lambda\in\mathbb{R}$.
\end{exmp}
The influence of the PDE parameter $\lambda$ and the ODE parameter $a$, which are known to rule the stability of both equations separately, is investigated in the sequel.

\subsubsection{Spectral analysis}

The characteristic equation~\eqref{eq:charac} is given by
\begin{equation}
    s-a+\frac{b}{\theta_i\mathrm{sinhc}(\sqrt{\frac{s-\lambda}{\nu}}\theta_i)} = 0.
\end{equation}

Considering real solutions $s=\mathrm{Re}(s)$, we obtain a sufficient condition of 
instability. For any $\frac{\lambda}{\nu}<(\frac{\pi}{\theta_i})^2$, the system is unstable if the coefficient $a$ satisfies
\begin{equation}\label{eq:condition}
    a > \frac{b}{\theta_i\mathrm{sinhc}(\sqrt{\frac{-\lambda}{\nu}}\theta_i)}.
\end{equation}

This condition means that there is no real positive intersection between $f_1(s)=-\mathrm{Re}(s)+a$ and $f_2(s)=\frac{b}{\theta_i\mathrm{sinhc}(\sqrt{\frac{\mathrm{Re}(s)-\lambda}{\nu}}\theta_i)}$. It is illustrated in Fig.~\ref{fig:spectral}, where the functions $f_1$ and $f_2$ are plotted in blue and magenta colors.
For instance, when $\lambda=0$, we find the instability condition $a\geq \frac{b}{\theta_i}$. For $b=0$ (without interconnection), both systems have to be stable separately.

\begin{rem}
Note that the limitation to $\frac{\lambda}{\nu}<(\frac{\pi}{\theta_i})^2$ is due to the fact we only consider the first branch of the above function $f_2$ in the spectral analysis.
\end{rem}

\begin{rem}
When $b<b_{max}$ with%=\frac{6\nu}{\theta_i}
\begin{equation}
\begin{aligned}
    \frac{1}{b_{max}} &= -\frac{\mathrm{d}}{\mathrm{d}s}\frac{1}{\theta_i\mathrm{sinhc}(\sqrt{\frac{s-\lambda}{\nu}}\theta_i)}(0),\\
    &=\frac{\theta_i}{2\nu}\frac{\cosh(\sqrt{\frac{-\lambda}{\nu}}\theta_i)-\mathrm{sinhc}(\sqrt{\frac{-\lambda}{\nu}}\theta_i)}{\mathrm{sinh}(\sqrt{\frac{-\lambda}{\nu}}\theta_i)^2} \underset{\lambda=0}{=} \frac{\theta_i}{6\nu}.
\end{aligned}
\end{equation}
the criterion~\eqref{eq:condition} is a necessary and sufficient stability condition applying Theorem~\ref{thm:spectral}. 
\end{rem} 

\begin{figure}
    \centering
    \includegraphics[width=7cm]{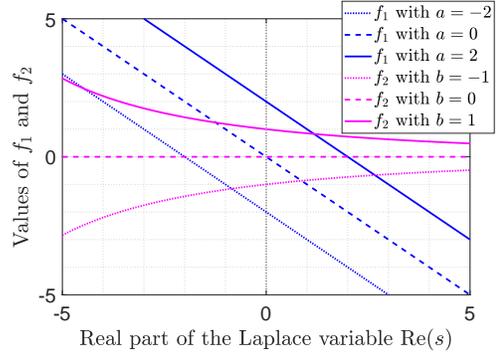}
    \caption{Spectral condition with $\lambda=0$, $\nu=\theta_i=1$.}%~\eqref{eq:condition}
    \label{fig:spectral}
\end{figure}

\subsubsection{Direct Lyapunov condition}

Theorem~\ref{thm2} can be used with Legendre polynomials or Fourier trigonometric functions normalized on the interval $[0,\theta_i]$. Firstly, we remark that the use of trigonometric functions is much more restrictive than Legendre polynomials. Indeed, only trigonometric functions can be generated on $[0,\theta_i]$ with Fourier basis~\cite{Powell1981}. 
Secondly, the instability condition depends on the order $n$. When $n$ increases, the certified instability regions grow. On Fig.~\ref{fig:1d}, we applied Theorem~\ref{thm2} with respect to point-wise values of $a$ and $\lambda$. 
The unstable points are represented with red points, whose size shrinks with the order $n$. %The hierarchy concerning the order $n$ is confirmed. 
%The size of unstable certified regions grows as the order increases. 
From the order $n=10$, we notice that there is no more improvement and that the estimate seems to converge to the unstable region colored in red.

\begin{figure}
    \centering
    %\begin{subfigure}{.95\linewidth}
    %\centering
        \includegraphics[width=7cm]{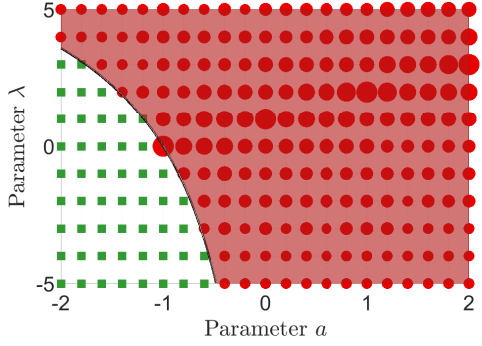}
    %    \subcaption{Order $n=5$.}
    %\end{subfigure}
    %\begin{subfigure}{.95\linewidth}
    %\centering
    %    \includegraphics[width=7cm]{Figures/lyapunov_order10.eps}
    %    \subcaption{Order $n=10$.}
    %\end{subfigure}
    \caption{Lyapunov direct condition with $b=-1$, $\nu=\theta_i=1$.}
    \label{fig:1d}%\label{fig:lyapunov}
\end{figure}

\subsubsection{Converse Lyapunov condition}

Consider the triplet~$(P,Q,T)$ as the solution of~\eqref{eq:PQT}, where $T$ is a piece-wise separable function~\cite{Stefanov2021}. This solution can be expressed as follows (see Appendix~\ref{app3} for calculation details)
\begin{equation}\label{eq:PQTsol1}
\begin{aligned}
    &P = \frac{\alpha\theta_i}{2} \mathrm{sinhc}(\sqrt{\frac{-\lambda}{\nu}}\theta_i),\;
    Q(\theta) = -\frac{\alpha\theta_i}{2} \frac{b}{\nu}\mathrm{sinhc}(\sqrt{\frac{-\lambda}{\nu}}\theta),\\
    &T(\theta_1,\theta_2)=\\
    &
    \left\{
    \begin{array}{l}
    \!\left(\!\cosh(\sqrt{\frac{-\lambda}{\nu}}\theta_1)+\beta\sinh(\sqrt{\frac{-\lambda}{\nu}}\theta_1)\!\right)\!\sinh(\sqrt{\frac{-\lambda}{\nu}}\theta_2) \frac{w}{2},\\
    \qquad \qquad \qquad  \qquad \qquad \qquad \qquad \text{for }0\leq\theta_2\leq\theta_1\leq 1,\\
    T(\theta_2,\theta_1),
    \qquad \qquad \qquad \qquad \quad\, \text{ otherwise},
    \end{array}
    \right.
    %f(\theta) = \cosh(\sqrt{\frac{-\lambda}{\nu}}\theta)+\beta\sinh(\sqrt{\frac{-\lambda}{\nu}}\theta)\\
    %g(\theta) = \frac{\sinh(\sqrt{\frac{-\lambda}{\nu}}\theta)}{\sqrt{\frac{-\lambda}{\nu}}},\\
\end{aligned}
\end{equation}
where 
\begin{equation}\label{eq:PQTsol2}
\begin{array}{l}
    \alpha = \left(b-a\theta_i\mathrm{sinhc}(\sqrt{\frac{-\lambda}{\nu}}\theta_i)\right)^{-1},\\
    \beta = \frac{a\cosh(\sqrt{\frac{-\lambda}{\nu}}\theta_i)}{b\sqrt{\frac{-\lambda}{\nu}}-a\sinh(\sqrt{\frac{-\lambda}{\nu}}\theta_i)},\\
    w = \frac{\alpha (\frac{b}{\nu})^2/\sqrt{\frac{-\lambda}{\nu}}}{\cosh(\sqrt{\frac{-\lambda}{\nu}}\theta_i)+\beta\sinh(\sqrt{\frac{-\lambda}{\nu}}\theta_i)}.
\end{array}
\end{equation}
%\frac{f^{\prime\prime}}{f}(\theta)=\frac{g^{\prime\prime}}{g}(\theta) = -\frac{\lambda}{\nu},\quad \forall \theta\in(0,\theta_i)\\
%g(0)=0
%f(\theta_i)a=b/\nu f^\prime(0)
It satisfies 
\begin{equation}
    \dot{\mathcal{V}}(\begin{bsmallmatrix}x\\z\end{bsmallmatrix})=-|x|^2- w\int_0^{\theta_i} |z(\theta)|^2\mathrm{d}\theta.
\end{equation}
For this particular case, the converse Lyapunov analysis reported by Corollary~\ref{cor:convP} can be implemented. For $w>0$, if $\alpha < 0$ then system~\eqref{eq:system} is unstable. It is equivalent to the spectral criterion~\eqref{eq:condition}.
The corresponding area is colored in red on Fig.~\ref{fig:1d}.

%On Fig.~\ref{fig:1d}, we can certify that Example~\ref{exmp:1d} is unstable for the parameters on the right of the black line and stable on the left.

\subsubsection{Comparison}

A comparison between the three previous results is reported in Fig.~\ref{fig:1d}. From one side, the spectral condition~\eqref{eq:condition} and the Lyapunov converse condition certify that the red area is unstable and that the black line shows the boundary between stable and unstable sets. From the other side, the Lyapunov direct condition with Legendre polynomials at order $n=10$ provides red points when it is unstable and green squares when we cannot conclude.% Lastly, the Lyapunov converse condition certify the the red area is unstable. %The three results corroborate to divide the space of parameters $(a,\lambda)\in[-2,2]\times[-5, 5]$ in a stable part and an unstable part.
%We materialize by a black line the limitation between stable and unstable sets.
%\begin{figure}[!ht]
%    \centering
%    \includegraphics[width=7cm]{Figures/figure1.eps}
%    \caption{Example~\ref{exmp:1d} with $b=-1$, $\nu=\theta_i=1$.}
%    \label{fig:1d}
%\end{figure}

For $b=-1$ and $\nu=\theta_i=1$, one can see that even if both the PDE and the ODE are stable, the interconnection might be unstable.
Intead, for $b=1$, the interconnection can stabilize unstable subsystems. %the opposite case could be enlightened.

Lastly, the spectral method and the converse Lyapunov method are restricted to simple cases and require the above analytical calculations (such as the computation of the rightmost characteristic root). The Lyapunov direct method is much more tractable and does not need pre-processing. It is also easily extendable to multi-dimensional reaction-diffusion PDE cases, other sets of system's parameters and other boundary conditions (Neumann or Robin types).

\subsection{Multivariable ODE example}

Consider a matrix case, where $A$ is Hurwitz and $\lambda>0$. It can be seen as a linear finite-dimensional controller at order $n_x=2$ of an unstable reaction-diffusion equation.% proportional-integral
\begin{exmp}\label{exmp:md}
    Consider system~\eqref{eq:system} with $A=\begin{bsmallmatrix}0&1\\-4&-4\end{bsmallmatrix}$, $B=\begin{bsmallmatrix}0\\\theta_i\end{bsmallmatrix}$, $C=\begin{bsmallmatrix}1&0\end{bsmallmatrix}$, $\nu>0$, $\lambda>0$, $\theta_i>0$, $\theta_o=(1-\alpha)\theta_i$ and $\alpha\in(0,0.5)$.
\end{exmp}
We focus on the instability phenomena occurred when the sensor or actuator location ($\theta_i$ or $\theta_o$) varies.

\subsubsection{Spectral analysis}

The characteristic equation~\eqref{eq:charac} is given by
\begin{equation}
    (s+2)^2-\frac{\cosh(\sqrt{\frac{s-\lambda}{\nu}}\alpha\theta_i)}{\mathrm{sinhc}(\sqrt{\frac{s-\lambda}{\nu}}\theta_i)} = 0.
\end{equation}

For real solutions $s=\mathrm{Re}(s)$, we obtain a sufficient condition of 
instability. For any $\theta_i<(\frac{\nu}{\lambda})^2\pi$, Example~\ref{exmp:md} is unstable if the coefficient $\alpha$ satisfies
\begin{equation}\label{eq:condition_md}
    \alpha < \frac{1}{\sqrt{\frac{-\lambda}{\nu}}\theta_i}\cosh^{-1}\left(4\mathrm{sinhc}(\sqrt{\frac{-\lambda}{\nu}}\theta_i)\right).
\end{equation}

This condition means that there is no real positive intersection between $f_3(s)=-(\mathrm{Re}(s)+2)^2$ and $f_4(s)=-\frac{\mathrm{cosh}(\sqrt{\frac{\mathrm{Re}(s)-\lambda}{\nu}}\alpha\theta_i)}{\mathrm{sinhc}(\sqrt{\frac{\mathrm{Re}(s)-\lambda}{\nu}}\theta_i)}$. It is illustrated on Fig.~\ref{fig:spectral_md}, where the functions $f_3$ and $f_4$ are plotted in blue and magenta colors. For $\nu=\lambda=1$, $\alpha=0.3$ and $\theta_i=3$, the pole $s\simeq 0.2$ yields an unstable the closed-loop system in Example~\ref{exmp:1d}. On Fig.~\ref{fig:md}, we depict condition~\eqref{eq:condition_md} as a black line. We can certify that Example~\ref{exmp:md} is unstable for the parameters on the right of the black line (red area) and stable on the left.

\begin{figure}[!t]
    \centering
    \includegraphics[width=7cm]{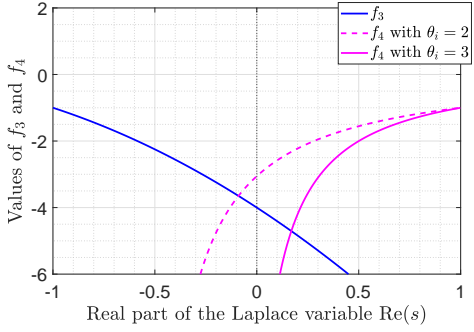}
    \caption{Example~\ref{exmp:md} with $\nu=\lambda=1$ and $\frac{\theta_o}{\theta_i}=0.7$ ($\alpha=0.3$).}
    \label{fig:spectral_md}
\end{figure}

\subsubsection{Direct Lyapunov condition}

The numerical condition $\Psi_n^+\succ 0$ and $\Pi_n^\top \Psi_n^-\Pi_n\prec 0$ in Theorem~\ref{thm2} is tested with Legendre polynomials at order $n=10$.
On Fig.~\ref{fig:md}, for point-wise values of $\theta_i$ and $\frac{\theta_o}{\theta_i}=1-\alpha$, a green square means that the condition succeeds, and a red point that it fails. Applying Theorem~\ref{thm2}, for red parameter values, the interconnected system in Example~\ref{exmp:md} is unstable. 

\begin{figure}[!t]
    \centering
    \includegraphics[width=7cm]{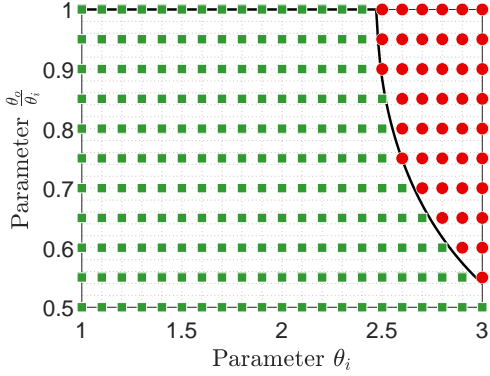}
    \caption{Example~\ref{exmp:md} with $\nu=\lambda=1$.}
    \label{fig:md}
\end{figure}

\subsubsection{Comparison}

Actually, for multivariable cases ($n_x>1$), the closed form of the converse Lyapunov functional satisfying~\eqref{eq:PQT} is unknown. Then, we only compare the spectral condition~\eqref{eq:condition_md} and the semidefinite programming condition given by Theorem~\ref{thm2}. Once again, as shown in Fig.~\ref{fig:md}, both results are similar. An unstable region is detected when $\theta_i$ or $\frac{\theta_o}{\theta_i}$ increase. It is also important to mention that the spectral condition enables to prove instability for a set of continuous parameters whereas the direct Lyapunov condition applies to a set of discrete parameters. %In terms of computational time, it is not comparable since for spectral the calculations are done upstream while for semidefinite programming the convex optimization algorithms are applied downstream.% (see YALMIP toolbox~\cite{Lofberg2004})% does the job using gradient descent methods (see Yalmip toolbox~\cite{Lofberg2004}). exist and\textbf{}

\section{Conclusions}
%, which can be seen as an open-loop plant, 
%, which can be seen as a finite-dimensional controller
This paper deals with a reaction-diffusion PDE coupled with an ODE. We developed instability tests to determine the role of several parameters on the behavior of the linear interconnected system. We propose three sufficient instability conditions: an eigenvalue test based on the spectral approach, a semidefinite programming test based on the direct Lyapunov approach, and a positive definiteness test based on the converse Lyapunov approach.

From a numerical point of view, we computed the sets of parameters for which the closed-loop system is unstable. We proved that the stability property of the interconnected system is independent of the stability of the PDE or ODE separately. %Furthermore, the determination of unstable regions might be easier than stable ones. %These results have finally wider perspectives as they can anticipate the design of a finite-dimensional controller or induce new stability analysis techniques.

Future work will extend the Lyapunov direct and converse methods to convex hulls of parameters and multi-dimensional PDE cases. Our theoretical results in terms of Riesz spectral decomposition could also lead to design controllability tests or control design strategies. % wider perspectives as they can anticipate the design of a finite-dimensional controller or induce new stability analysis techniques. %In the longer term, we would like to tackle the control design problem using the spectral or temporal techniques developed here.

\appendices

\section{Expression of some characteristic functions}\label{app1}

Assuming that $\begin{bsmallmatrix}x_k(t)\\z_k(t,\theta)\end{bsmallmatrix}=e^{s_kt}\begin{bsmallmatrix}X_k\\Z_k(\theta)\end{bsmallmatrix}$ is a non null solution of system~\eqref{eq:system} implies that %and $s_k\notin \left\{-\nu(\frac{k\pi}{\theta_i})^2+\lambda\right\}_{k\in\mathbb{N}}\cup \sigma(A)$. It implies that
\begin{equation}\label{eq:systemcharac}
    \left\{
    \begin{aligned}
        s_k X_k &= A X_k + B \partial_\theta Z_k(\theta_o),\\
        s_k Z_k(\theta) &= \nu Z_k^{\prime \prime}(\theta) + \lambda Z_k(\theta),\quad  \forall \theta\in(0,\theta_i),\\
        Z_k(0) &= CX_k, \quad Z_k(\theta_i) = 0.
    \end{aligned}
    \right.
\end{equation}
From the ODE part, the vector $X_k$ satisfies %since $\mathrm{adj}(M) M =\mathrm{det}(M)$ for any square matrix $M$, we have
\begin{equation}
    \mathrm{det}(s_k I_{n_x}-A) X_k = \mathrm{adj}(s_k I_{n_x}-A)B \partial_\theta Z_k(\theta_o).
\end{equation}
From the PDE part, the function $Z_k$ satisfies%since $\sinh\left(\sqrt{\frac{s_k-\lambda}{\nu}}\theta_i\right)^{-1}$ exists, 
\begin{equation*}
    \sinh\left(\sqrt{\frac{s_k-\lambda}{\nu}}\theta_i\right)Z_k(\theta) = \sinh\left(\sqrt{\frac{s_k-\lambda}{\nu}}(\theta_i-\theta)\right)CX_k.
\end{equation*}
By derivation and evaluation at $\theta_o$, the coupling gives %$\partial_\theta Z_k(\theta_o)\neq 0$ if and only if
\begin{equation}
    \mathrm{det}(s_k I_{n_x}\!-\!A) -C\mathrm{adj}(s_k I_{n_x}\!-\!A)BH(s_k)=0 \Leftrightarrow \Delta(s_k) = 0.
\end{equation}
% Step 1 : det(det - BHCadj)
% Step 2 :det(M)det(N)=det(MN) avec * det(sI-A) / det(sI-A)
% donne det(sI-A -BHC adj(sI-A)(sI-A))=det(sI-A -BHC)
Then, for any $u=\frac{C\mathrm{adj}(s_kI_{n_x}-A)B}{\mathrm{det}(s_kI_{n_x}-A)\sinh\left(\sqrt{\frac{s_k-\lambda}{\nu}}\theta_i\right)}\partial_\theta Z_k(\theta_o)$ in $\mathbb{C}$ and for $s_k$ solution of $\Delta(s_k)=0$, we obtain that %$\begin{bsmallmatrix}X_k\\Z_k(\theta)\end{bsmallmatrix}=F_k $ where
\begin{equation}
    \begin{bsmallmatrix}X_k\\Z_k(\theta)\end{bsmallmatrix} = \begin{bsmallmatrix} \mathrm{adj}(s_kI_{n_x}-A)B \frac{\sinh\left(\sqrt{\frac{s_k-\lambda}{\nu}}\theta_i\right)}{C\mathrm{adj}(s_kI_{n_x}-A)B} \\ \sinh\left(\sqrt{\frac{s_k-\lambda}{\nu}}(\theta_i-\theta)\right)\end{bsmallmatrix} u,
\end{equation}
solve~\eqref{eq:systemcharac}. To conclude, $F_k:=\begin{bsmallmatrix}X_k\\Z_k(\theta)\end{bsmallmatrix}\mathbf{i}$ is the normalized characteristic function of $\mathcal{A}$ in $\mathcal{D}$ associated to the characteristic root $s_k$ solution of~\eqref{eq:charac}.

\section{Converse Lyapunov functional}

\subsection{Kernels equation}\label{app2}

Consider $(P,Q,T)$ in $\mathbb{S}^{n_x}_+\times L^2(0,\theta_i)^{n_x}\times L^2((0,\theta_i)\times(0,\theta_i))$ such that $T(\theta_1,\theta_2)=T(\theta_2,\theta_1)$ and that the Lyapunov functional $\mathcal{V}$ in~\eqref{eq:lyap} satisfies along the trajectories of system~\eqref{eq:system}
\begin{equation}
    \frac{1}{2}\dot{\mathcal{V}}(x,z) = -\frac{1}{2}|x|^2 - \frac{w}{2} \int_0^{\theta_i} |z(\theta)|^2\mathrm{d}\theta.
\end{equation}
Applying integration by parts to the expression of $\dot{\mathcal{V}}$ in~\eqref{eq:lyapder} leads to
\begin{equation*}
    \begin{aligned}
    & x^\top\!PAx \!+\! x^\top PB \partial_\theta z(\theta_i) \!+\! \partial_{\theta}z^\top\!(\theta_i)\!\!\int_0^{\theta_i}\!\!\! B^\top Q(\theta) z(\theta) \mathrm{d}\theta\\
    & \!+\! x^\top\! \int_0^{\theta_i}\!\!\! (A^\top\!+\!\nu\partial_{\theta\theta}\!+\!\lambda I_{n_x}) Q(\theta)z(\theta)\mathrm{d}\theta \!+\! \nu x^\top\!\left[Q(\theta)\partial_\theta z(\theta)\right]_0^{\theta_i}\\
    & \!-\! \nu x^\top\!\left[\partial_\theta Q(\theta)z(\theta)\right]_0^{\theta_i} 
    \!+\! \nu \int_0^{\theta_i} \!\!\! z^\top\!(\theta)\left[T(\theta,\tau)\partial_{\tau}z(\tau)\right]_0^{\theta_i}\mathrm{d}\theta \\
    & \!-\! \nu \int_0^{\theta_i} \!\!\! z^\top\!(\theta)\left(\left[\partial_{\tau}T(\theta,\tau)z(\tau)\right]_0^{\theta}\!+\!\left[\partial_{\tau}T(\theta,\tau)z(\tau)\right]_{\theta}^{\theta_i}\right)\mathrm{d}\theta \\
    & \!+\! \int_0^{\theta_i}\!\!\!\int_0^{\theta_i}\!\!\!z^\top(\theta_1)(\nu\partial_{\theta_2\theta_2}+\lambda)T(\theta_1,\theta_2)z(\theta_2)\mathrm{d}\theta_1\mathrm{d}\theta_2\\
    &= -\frac{1}{2}|x|^2 - \frac{w}{2} \int_0^{\theta_i} |z(\theta)|^2\mathrm{d}\theta.
    \end{aligned}
\end{equation*}
Boundary conditions $z(0)=Cx$ and $z(\theta_i)=0$ allow to simplify in% the previous expression simplifies to
\begin{equation*}
    \begin{aligned}
    & x^\top\!(PA \!+\! \nu Q^\prime(0)C \!+\! \frac{1}{2})x \!+\! x^\top (PB\!+\!\nu Q(\theta_i)) \partial_\theta z(\theta_i)\\
    & \!+\! x^\top\! \int_0^{\theta_i}\!\!\! \left(\!(A^\top\!+\!\nu\partial_{\theta\theta}\!+\!\lambda I_{n_x}) Q(\theta) \!+\!  \nu C^\top \!\!\underset{\tau\to 0}{\mathrm{lim}}\partial_\tau T^\top(\theta,\tau) \!\right)\!z(\theta) \mathrm{d}\theta\\
    %\!+\! \nu \! \int_0^{\theta_i} \!\!\! z^\top\!(\theta) \partial_\tau T(\theta,0)Cx\mathrm{d}\theta
    & \!+\! \partial_{\theta}z^\top\!(\theta_i) \!\int_0^{\theta_i} \!\!\! (BQ(\theta)  \!+\! \nu T^\top\!(\theta,\theta_i))  z(\theta) \mathrm{d}\theta \\
    & \!-\! \nu x^\top Q(0)\partial_\theta z(0) \!-\! \nu\!\int_0^{\theta_i} \!\!\! z^\top (\theta) T(\theta,0) \partial_{\theta}z(0)   \mathrm{d}\theta \\
    & \!-\! \nu \! \int_0^{\theta_i} \!\!\! z^\top\!(\theta)\left(\underset{\tau\to\theta^-}{\mathrm{lim}}\partial_{\tau}T(\theta,\tau)\!-\!\underset{\tau\to\theta^+}{\mathrm{lim}}\partial_{\tau}T(\theta,\tau) \!+\! \frac{w}{2}\right)z(\theta)\mathrm{d}\theta \\
    & \!+\! \int_0^{\theta_i}\!\!\!\int_0^{\theta_i}\!\!\!z^\top(\theta_1)(\nu\partial_{\theta_2\theta_2}+\lambda)T(\theta_1,\theta_2)z(\theta_2)\mathrm{d}\theta_1\mathrm{d}\theta_2 = 0.
    \end{aligned}
\end{equation*}
Symmetric properties satisfied by the matrix~$P$ and the function~$T$ allow us to conclude that the constraints in~\eqref{eq:PQT} need to be satisfied.

\subsection{Kernels solution}\label{app3}

Assuming that function $T$ is a separable function on the triangle $\{(\theta_1,\theta_2)\in[0,1]^2\,|\,\theta_1\geq \theta_1\}$, the PDE part~\eqref{eq:PQT4} leads to
\begin{equation}
\begin{aligned}
    T(\theta_1,\theta_2)\!=\!&\left(\!\beta_1\cosh(\sqrt{\frac{-\lambda}{\nu}}\theta_1)+\beta_2\sinh(\sqrt{\frac{-\lambda}{\nu}}\theta_1)\!\right)\!\\
    &\quad\left(\!\beta_3\cosh(\sqrt{\frac{-\lambda}{\nu}}\theta_2)+\beta_4\sinh(\sqrt{\frac{-\lambda}{\nu}}\theta_2)\!\right),
\end{aligned}
\end{equation}
where $\beta_1$, $\beta_2$, $\beta_3$ and $\beta_4$ are some scalars to be fixed. Normalizing $\beta_1~=~1$, the boundary conditions~\eqref{eq:PQT3},\eqref{eq:PQT5},\eqref{eq:PQT6} give
\begin{equation}
\begin{aligned}
    P &= (\frac{\nu}{b})^2 T(\theta_i,\theta_i),&& Q(\theta) = -\frac{\nu}{b}T(\theta_i,\theta),\\
    \beta_3 &= 0,&& \beta_4 = \frac{w}{2}.
\end{aligned}
\end{equation}
Moreover, the ODE part~\eqref{eq:PQT2} allows to fix 
\begin{equation}
    \beta_2 = \frac{a\cosh(\sqrt{\frac{-\lambda}{\nu}}\theta_i)}{b\sqrt{\frac{-\lambda}{\nu}}-a\sinh(\sqrt{\frac{-\lambda}{\nu}}\theta_i)}.
\end{equation}
Then, the last constraint~\eqref{eq:PQT1} imposes
\begin{equation}
    aT(\theta_i,\theta_i) -b\partial_{\theta_2}T(\theta_i,0) = -\frac{(\frac{b}{\nu})^2}{2},
\end{equation}
which means that
\begin{equation}
    w = \frac{(\frac{b}{\nu})^2/\sqrt{\frac{-\lambda}{\nu}}}{\!\left(\!\cosh(\sqrt{\frac{-\lambda}{\nu}}\theta_i)\!+\!\beta\sinh(\sqrt{\frac{-\lambda}{\nu}}\theta_i)\!\right)\!\left(\!b-a\theta_i\mathrm{sinhc}(\sqrt{\frac{-\lambda}{\nu}}\theta_i)\!\right)\!}.
\end{equation}
Lastly, denoting 
\begin{equation}
    \alpha= \frac{\cosh(\sqrt{\frac{-\lambda}{\nu}}\theta_i)+\beta\sinh(\sqrt{\frac{-\lambda}{\nu}}\theta_i)}{(\frac{b}{\nu})^2/\sqrt{\frac{-\lambda}{\nu}}}w,\qquad \beta=\beta_2,
\end{equation}
we recover the solution provided in~\eqref{eq:PQTsol1}-\eqref{eq:PQTsol2}.

\bibliographystyle{plain}
\bibliography{autosam} 

\end{document}